\documentclass{amsart}
\usepackage{amsfonts}
\usepackage{color}
\usepackage{graphicx}
\usepackage[a4paper,bookmarksnumbered,colorlinks, linkcolor=blue, citecolor=red, pagebackref, bookmarks, breaklinks]{hyperref}

\newtheorem{thm}{Theorem}[section]

\newtheorem{lema}[thm]{Lemma}
\newtheorem{prop}[thm]{Proposition}
\theoremstyle{definition}

\theoremstyle{remark}

\newtheorem{rem}[thm]{Remark}
\numberwithin{equation}{section}
\newcommand{\R}{\mathbb R}

\newcommand{\Z}{\mathbb Z}

\newcommand{\ve}{\varepsilon}

\newcommand{\lam}{\lambda}

\newcommand{\cd}{\rightharpoonup}

\newcommand{\cf}{\rightarrow}

\begin{document}
\title{Convergence rates in a weighted Fu\u{c}ik problem}
\author[Ariel M Salort]{Ariel M. Salort}
\address{Departamento de Matem\'atica
 \hfill\break \indent FCEN - Universidad de Buenos Aires and
 \hfill\break \indent   IMAS - CONICET.
\hfill\break \indent Ciudad Universitaria, Pabell\'on I \hfill\break \indent   (1428)
Av. Cantilo s/n. \hfill\break \indent Buenos Aires, Argentina.}
\email{asalort@dm.uba.ar}

\begin{abstract}
In this work we consider the Fu\u{c}ik problem for a family of weights depending on $\ve$ with Dirichlet and Neumann boundary conditions. We study the homogenization of the spectrum. We also deal with the special case of periodic homogenization and we obtain the rate of convergence of the first non-trivial curve of the spectrum.
\end{abstract}



\maketitle

\section{Introduction}

Given a bounded domain  $\Omega$  in $\R^N$, $N\geq 1$ we study the asymptotic behavior as $\ve \cf 0$ of the spectrum of the following asymmetric elliptic problem
\begin{align} \label{P1}
-\Delta_p u_\ve=\alpha_\ve m_\ve(u_\ve^+ )^{p-1} - \beta_\ve n_\ve(u_\ve^- )^{p-1} \quad &\textrm{ in } \Omega
\end{align}
either with homogeneous Dirichlet or Neumann boundary conditions.

Here, $\Delta_p u=div(|\nabla u|^{p-2}\nabla u)$ is the $p-$Laplacian with $1<p<\infty$ and $u^\pm:=\max\{\pm u, 0\}$. The parameters $\alpha_\ve$ and $\beta_\ve$ are reals and depending on $\ve>0$.
We assume that the family of weight functions $m_\ve$ and $n_\ve$ are positive and uniformly bounded away from zero.

For a moment let us focus  problem \eqref{P1} for fixed $\ve>0$  with positive weights $m(x),n(x)$:
\begin{align} \label{P2}
-\Delta_p u=\alpha m(x)(u^+ )^{p-1} -\beta n(x)(u^- )^{p-1} \quad &\textrm{ in } \Omega
\end{align}
with Dirichlet or Neumann boundary conditions.

Consider the Fu\u{c}ik spectrum defined as the set
$$\Sigma(m,n):=\{(\alpha ,\beta) \in \R^2\colon \eqref{P2} \textrm{ has a nontrivial solution} \}.$$

Let us observe that when $r=n=m$ and $\lam=\alpha=\beta$, equation \eqref{P2} becomes
\begin{align} \label{Plap}
-\Delta_p u= \lam r |u|^{p-2}u  \quad &\textrm{ in } \Omega
\end{align}
with Dirichlet or Neumann boundary conditions, which is the eigenvalue problem for the $p-$Laplacian. These has been widely studied. See for instance \cite{ANA,CUE,HT,DGT} for more information.

It follows immediately  that $\Sigma$ contain the lines $\lam_1(m)\times \R$ and $\R \times \lam_1(n)$. For this reason, we denote by $\Sigma^*=\Sigma^*(m,n)$ the set $\Sigma$ without these trivial lines. Observe that if $(\alpha,\beta) \in \Sigma^*$ with $\alpha \geq 0$ and $\beta \geq 0$ then $\lam_1(m)<\alpha$ and $\lam_1(n) < \beta$.

The study of problem \eqref{P2} with Dirichlet boundary conditions  have a long history that we briefly describe  below. The one-dimensional case with positive constant coefficients (i.e., $m,n\in \R^+$ and $p=2$) was studied in the 1970s by Fu\u{c}ik \cite{FUC}  and Dancer \cite{DAN} in connection with jumping nonlinearities. Properties and descriptions of the first non-trivial curve on the spectrum of \eqref{P2} on $\R^N$ for the general case ($p\neq 2$) without weights can be found in Cuesta, de Figueiredo and Gossez \cite{CUE}, Dancer and Perera \cite{DAN-PER}, Dr\'abek and Robinson \cite{DRA-ROB}, Perera \cite{PER}.

The case with positive weights $m(x)$ and $n(x)$ was recently studied, see for instance Rynne and Walter \cite{RYN}, Arias and Campos \cite{AR1}, Drabek \cite{DRA}, Reichel and Walter \cite{RE1}. For  indefinite weights $m(x)$ and $n(x)$ see Alif and Gossez\cite{ALIF1}, Leadi and Marcos \cite{LIAM}.

The main problem one address is to obtain a description as accurate as possible of the set $\Sigma^*$. In the one-dimensional case, $p=2$, without weights this description is obtained in a precise manner: the spectrum is made of a sequence of hyperbolic like curves in $\R^+\times\R^+$, see for instance \cite{FUCIKLIBRO}. When $m(x)$ and $n(x)$ are non-constants weights,  in \cite{ALIF1} it is  proved a characterization of the spectrum in terms of the so-called zeroes-functions.

In $\R^N$ with $N>1$ and Dirichlet boundary conditions, only a full description of the first nontrivial curve of $\Sigma$ is known, which we will denote by $\mathcal{C}_1=\mathcal{C}_1(m,n)$.

Assuming that the weight functions $m,n$ are positive and uniformly bounded, in \cite{ARI} (see Theorem 33) is proved that $\mathcal C_1$ can be characterized by
\begin{equation}
 \mathcal{C}_1=\{(\alpha(s),\beta(s)) , s\in\R^+\}
\end{equation}
where $\alpha(s)$ and $\beta(s)$ are continuous functions defined by
\begin{equation} \label{caract}
\alpha(s)=c(m,sn),\quad \beta(s)=s\alpha(s)
\end{equation}
and $c(\cdot,\cdot)$ is given by
\begin{equation} \label{defc}
 c(m,n)=\inf_{\gamma \in \Gamma} \max_{u\in \gamma(I)} \frac{A(u)}{B(u)}.
\end{equation}
where $I:=[-1,+1]$. Here, the functionals $A$ and $B$ are given by
\begin{equation} \label{funcionales}
A(u)=\int_\Omega |\nabla u|^p dx,\quad  B_{m,n}=\int_\Omega m(x) (u^+)^p + n(x) (u^-)^p dx,
\end{equation}
with
$$\Gamma=\{\gamma \in C([-1,+1],W^{1,p}_0(\Omega)):\gamma(-1)\geq 0 \textrm{ and } \gamma(1) \leq 0\}.$$
In \cite{ARI} (see Proposition 34) some important properties of the functions $\alpha(s)$ and $\beta(s)$ are proved. Namely, both $\alpha(s)$ and $\beta(s)$ are continuous, $\alpha(s)$ is strictly decreasing and $\beta(s)$ is strictly increasing. One also has that $\alpha(s) \cf +\infty$ if $s\cf 0$ and $\beta(s) \cf +\infty$ is $s\cf +\infty$.

\bigskip

Having defined these previous concepts and definitions, let us back to problem \eqref{P1}.

Homogenization of the spectrum of elliptic operators was extensively studied in the last years. The case of the eigenvalues of the weighted $p-$Laplacian operator in periodic settings, i.e., $-\Delta_p u_\ve = \rho_\ve |u_\ve|^{p-2}u_\ve$ with Dirichlet boundary conditions and $\rho_\ve=\rho(\tfrac{x}{\ve})$, $\rho$ being a $Q$-periodic function with $Q$ the unit cube in $\R^N$, together with a family of more general problems it was studied for instance by \cite{Zua2},\cite{Zua1},\cite{1DIM},\cite{Kenig},\cite{Ol} in the linear case ($p=2$) and by \cite{Con},\cite{DF1992},\cite{NDIM} in the non-linear case ($p\neq 2$).

Up to our knowledge, no  investigation was made in the homogenization and rates of convergence of the Fu\u{c}ik Spectrum. We are interested in studying the behavior as $\ve\cf 0$ of problem \eqref{P1} when $m_\ve(x)$ and $n_\ve(x)$ are general functions depending on $\ve$, and in the special case of rapidly oscillating periodic functions, i.e., $m_\ve(x)=m(x/\ve)$ and $n_\ve(x)=n(x/\ve)$ for two $Q-$periodic functions $m,n$ uniformly bounded away from zero (see assumptions \eqref{cotas}), $Q$ being the unit cube of $\R^N$.

Our main aim is to study the limit as $\ve \cf 0$ of the first nontrivial curve in the spectrum $\Sigma_\ve:=\Sigma(m_\ve,n_\ve)$, say $\mathcal{C}_1^\ve=\{(\alpha_\ve(s),\beta_\ve(s)) , s\in\R^+\}$. We wonder: there exists a limit curve  $ \mathcal{C}_1=\{(\alpha_0(s),\beta_0(s)) , s\in\R^+\}$ such that
$$ \mathcal{C}_1^\ve \cf \mathcal{C}_1, \quad \textrm{ as } \ve \cf 0 \textrm{ ? }$$
Can this limit curve be characterized like a curve of a limit problem? We will see that the answer is positive.

Therefore, a natural question arises: can the rate of convergence of $\mathcal{C}_1^\ve$ be estimated? I.e., can we give an estimate  of the remainders
$$|\alpha_\ve(s) -\alpha_0(s)| \quad \textrm{ and } \quad |\beta_\ve(s) -\beta_0(s)|?$$
We give positive answers to these questions in the periodic setting. In fact, in Theorem \ref{teo_2} we obtain the bounds
$$|\alpha_\ve(s) - \alpha_0(s)| \leq c  (1+s) \tau(s) \ve   , \quad |\beta_\ve(s) - \beta_0(s)| \leq c  s(1+s) \tau(s)  \ve, \quad s\in\R^+$$
where $c$ is a constant fully determined which is independent of $s$ and $\ve$, and $\tau$ is a explicit function depending only of $s$ (see \eqref{ftau.2}).

Particularly, for the limit values of the coordinates, we get
$$|\alpha^\infty_\ve - \alpha^\infty_0| \leq c \ve   , \quad |\beta^0_\ve - \beta^0_0| \leq c  \ve$$
where $\alpha^\infty_\ve=\displaystyle \lim_{s\cf \infty}\alpha_\ve(s)$,  $\alpha^\infty_0=\displaystyle  \lim_{s\cf \infty}\alpha_0(s)$,   $\beta^0_\ve=\displaystyle \lim_{s\cf \infty}\beta_\ve(s)$,  $\beta^0_0=\displaystyle  \lim_{s\cf \infty}\beta_0(s)$ and $c$ is independent of $s$ and $\ve$.

\section{The results}
Let $\Omega\subset \R^N$ be a bounded domain and $\ve$ a real positive number. We consider functions  $m_\ve,n_\ve$  such that for constants $m_-\leq m_+, n_- \leq n_+$
\begin{equation} \label{cotas}
0<m_-\leq m_\ve(x) \leq m_+ \leq +\infty \quad  \textrm{ and }  \quad 0<n_-\leq n_\ve(x) \leq n_+ \leq +\infty.
\end{equation}
Also, we assume that there exist functions $m_0(x)$ and $n_0(x)$ satisfying \eqref{cotas} such that, as $\ve\cf 0$,
\begin{align} \label{limi}
\begin{split}
\begin{array}{ll}
  m_\ve(x)\cd m_0(x) & \quad\textrm{ weakly* in }L^\infty(\Omega) \\
  n_\ve(x)\cd n_0(x) & \quad\textrm{ weakly* in }L^\infty(\Omega).
\end{array}
\end{split}
\end{align}
First, we address the problem with Dirichlet boundary conditions.

When $\ve \cf 0$  the natural limit problem for \eqref{P1} is the following
\begin{align} \label{pron1limite.gral}
\begin{cases}
-\Delta_p u_0=\alpha_0 m_0(x) (u_0^+ )^{p-1} - \beta_0 n_0(x)(u_0^- )^{p-1} \quad &\textrm{ in } \Omega \\
u_0=0 \quad &\textrm{ on } \partial \Omega
\end{cases}
\end{align}
where $m_0$ and $n_0$ are given in \eqref{limi}.

The main result is the following:
\begin{thm} \label{teo_2.gral}
Let $m_\ve,n_\ve$ satisfying \eqref{cotas} and \eqref{limi}. Then the first non-trivial curve of problem \eqref{P1}
$$\mathcal{C}_\ve:=\mathcal{C}_1(m_\ve,n_\ve)=\{\alpha_\ve(s),\beta_\ve(s), s\in\R^+\}$$
converges to the first non-trivial curve  of the limit problem \eqref{pron1limite.gral}
$$\mathcal{C}:=\mathcal{C}_1(m_0,n_0)=\{\alpha_0(s),\beta_0(s), s\in\R^+\}$$
as $\ve \cf 0$ in the sense that $\alpha_\ve(s)\cf \alpha_0(s)$ and  $\beta_\ve(s)\cf \beta_0(s)$ $\forall s\in\R^+$.
\end{thm}

\begin{rem} \label{remconv}
Let us consider the weighted $p-$Laplacian problem
\begin{align} \label{plap.ve}
\begin{cases}
-\Delta_p u= \lam r_\ve(x) |u|^{p-2}u  \quad &\textrm{ in } \Omega \\
u=0 \quad &\textrm{ on } \partial \Omega
\end{cases}
\end{align}
where $r_\ve$ is a function such that $r_\ve(x) \cd r(x)$ weakly* in $L^\infty(\Omega)$ as $\ve$ tends to zero. It is well-known that the first eigenvalue of \eqref{plap.ve} converges to the first eigenvalue of the  $p-$Laplacian equation with weight $r(x)$, see for instance \cite{Con}. The fact that the trivial lines of $\Sigma_\ve$ are defined by $\lam_1(m_\ve) \times \R$ and $\R\times \lam_1(n_\ve)$ it allows us to affirm the convergence of the trivial lines to those of the limit problem.
\end{rem}

\begin{rem}
Using the variational characterization of the second (variational) eigenvalue of \cite{ARI}, Theorem \ref{teo_2.gral} implies the convergence of the second (variational) eigenvalue of \eqref{plap.ve} to those of the limit problem, which recover result recently proved in \cite{NDIM} for the case of the weighted $p-$Laplacian. However, the results in \cite{NDIM} consider a more general class of quasilenar operators and  $\ve$-dependence on the operator as well.
\end{rem}

In the important case of periodic homogenization, i.e., when $m_\ve(x)=m(x/\ve)$ and $n_\ve(x)=n(x/\ve)$ where $m$ and $n$ are $Q-$periodic functions, $Q$ being the unit cube in $\R^N$, we have that $m_0=\bar m$ and $n_0=\bar n$ are real numbers  given by the averages of $m$ and $n$ over $Q$, respectively. Consequently, the limit problem \eqref{pron1limite.gral} becomes
\begin{align} \label{pron1limite}
\begin{cases}
-\Delta_p u_0=\alpha_0 \bar m (u_0^+ )^{p-1} - \beta_0 \bar n (u_0^- )^{p-1} \quad &\textrm{ in } \Omega \\
u_0=0 \quad &\textrm{ on } \partial \Omega.
\end{cases}
\end{align}
In this case, besides the convergence of the curves given in  Theorem \ref{teo_2.gral} and Remark \ref{remconv},  we obtain the  convergence rates.

First, by using the variational characterization of the first eigenvalue of \eqref{plap.ve} we analyze the trivial lines of $\Sigma_\ve$:
\begin{thm} \label{teo_1}
Let $m_\ve,n_\ve$ weights satisfying \eqref{cotas} and \eqref{limi} given in terms of $Q-$periodic functions $m,n$ in the form $m_\ve(x)=m(\tfrac{x}{\ve})$ and $n_\ve(x)=n(\tfrac{x}{\ve})$. Let us denote by $\lam_1(m_\ve)$, $\lam_1(n_\ve)$, $\lam_1(\bar m)$ and $\lam_1(\bar n)$ to the first eigenvalue  of equation \eqref{plap.ve} with weight $m_\ve$, $n_\ve$, $\bar m$ and $\bar n$, respectively. Then
$$|\lam_1(m_\ve)-\lam_1(\bar m)|\leq C_m\ve, \qquad |\lam_1(n_\ve)-\lam_1(\bar n)|\leq C_n\ve,$$
with $C_m$ given by
$$C_m= pc_1 \|m-\bar m\|_{L^\infty(\R^N)} (m_+)^{1/p} (m_-)^{-\frac{1}{p}-2} \mu_1^{\frac{1}{p}+1},$$
where $\mu_1$ is the first eigenvalue of the Dirichlet $p-$Laplacian and $c_1 \leq \sqrt{N}/2$.
\end{thm}
\begin{rem}
From Theorem \ref{teo_1} it follows the  convergence rates of the trivial lines of $\Sigma_\ve$: if $p_\ve \in \lam_1(m_\ve) \times \R$, we get
$$|p_\ve-p_0|\leq C_m\ve,$$
where $p_0$ belongs to the line $\lam_1(\bar m) \times \R$. Analogously for $p_\ve \in \R \times \lam_1(n_\ve) $.
\end{rem}

Related to the first nontrivial curve of  $\Sigma_\ve$ we obtain:

\begin{thm} \label{teo_2}
Under the same considerations of Theorem \ref{teo_2.gral}, if the weights $m_\ve$ and $n_\ve$ are given in terms of $Q-$periodic functions $m,n$ in the form $m_\ve(x)=m(\tfrac{x}{\ve})$ and $n_\ve(x)=n(\tfrac{x}{\ve})$,  for each $s\in \R^+$, we have the following estimates
\begin{align} \label{ect1}
|\alpha_\ve(s) - \alpha_0(s)| \leq c  (1+s) \tau(s) \ve   , \quad |\beta_\ve(s) - \beta_0(s)| \leq c s(1+s) \tau(s) \ve
\end{align}
where $c $ is given explicitly by
$$ p c_1 c_p^{p-1} \max\{\|m-\bar m\|_{L^\infty(\R^N)},\|n-\bar n\|_{L^\infty(\R^N)}\} (\min\{m_-^{-1},n_-^{-1}\}\mu_2)^2$$
where $c_1$ and $c_p$ are the Poincar\'e's constant in $L^1(Q)$ and $L^p(\Omega)$, respectively, $\mu_2$ is the second Dirichlet $p-$Laplacian eigenvalue in $\Omega$ and $\tau$ is defined by
\begin{equation} \label{ftau.2}
\tau(s)=
\begin{cases}
1 &\quad s \geq 1 \\
s^{-2} &\quad s < 1.
\end{cases}
\end{equation}
\end{thm}
\medskip

\begin{rem}

According to Proposition 34 and Proposition 35 in \cite{ARI}, when $p\leq N$ the limits of $\alpha_\ve(s), \alpha_0(s)$ as $s\cf \infty$ and $\beta_\ve(s),\beta_0(s)$  as $s\cf 0$ can be characterized in terms of the first eigenvalues of weighted $p-$Laplacian problems. Moreover, $\lim_{s\cf\infty} \alpha_\ve(s)= \lam_1(m_\ve)$ and $\lim_{s\cf 0} \beta_\ve(s)= \lam_1(n_\ve)$. Similarly for $\alpha_0$ and $\beta_0$. Consequently, by using the estimates obtained in Theorem \ref{teo_1}, it is easy to compute the convergence rates in the limit cases when the periodic case is considered, namely
\begin{align*}
&\lim_{s\cf \infty} |\alpha_\ve(s) - \alpha_0(s)| = |\lam_1(m_\ve)-\lam_1(\bar m)| \leq C_m\ve, \\
&\lim_{s\cf 0} |\beta_\ve(s) - \beta_0(s)| = |\lam_1(n_\ve)-\lam_1(\bar n)| \leq C_n\ve.
\end{align*}
\end{rem}
\medskip
Now we focus our attention on the Neumann boundary conditions case, i.e., we study the homogenization of the spectrum of the Fu\u{c}ik problem
\begin{align}  \label{ne}
\begin{cases}
-\Delta_p u=\alpha m(x)(u^+ )^{p-1} -\beta n(x)(u^- )^{p-1} \quad &\textrm{ in } \Omega \\
\tfrac{\partial u}{\partial \eta }=0 \quad &\textrm{ on } \partial \Omega.
\end{cases}
\end{align}
where $\partial u/\partial \eta=\nabla u\cdot \eta$ denotes the unit exterior normal.
In Section \ref{sec.neu} we study the limit problem associated to \eqref{ne} and the homogenization of the first non-trivial curve of its spectrum. We obtain similar results to the Dirichlet case:
in Theorem \ref{teo_2.gral.neu} we study the convergence in general settings; in Theorem \ref{teo_2.neu} we deal with the periodic case, obtaining  convergence rates similar to those of Theorem \ref{teo_2}.

\section{Proof of the Dirichlet results}
We begin with the proof of the Theorem \ref{teo_1}. For that, we will use a technical result proved in \cite{NDIM} that is essential to estimate the rate convergence of the eigenvalues since allows us to replace an integral involving a rapidly oscillating function with one that involves its average in the unit cube.

\begin{thm} [Theorem 3.4 in \cite{NDIM}.] \label{teo_n_dim}
Let $\Omega$ be a bounded domain in $\R^N$, $N\geq 1$. Let $g\in L^\infty(\R^N)$ be a $Q-$periodic function, being $Q=[0,1]^N$ the unit cube in $\R^N$, such that $0<g^- \leq g \leq g^+<+\infty$ for $g^\pm$ constants. Then
\begin{align} \label{ecuteo}
 \left| \int_\Omega (g(\tfrac{x}{\ve})- \bar{g}) |u|^p \right| \leq pc_1 \|g-\bar g\|_{L^\infty(\R^N)} \ve \| u\|_{L^p(\Omega)}^{p-1} \|\nabla u\|_{L^p(\Omega)}
\end{align}
for every $u\in W^{1,p}_0(\Omega)$ where $1< p < +\infty$, $\Omega\subset \R^N$ bounded domain and $\bar{g}:=\int_Q g$. Here, $c_1$ is the optimal constant in Poincar\'e's inequality in $L^1(Q)$ which satisfies $c_1 \leq \sqrt{N}/2$.
\end{thm}
\begin{rem} \label{teo_n_dim_rem}
Sometimes it will be useful to use an inequality involving only the gradient.
  By using Poincar\'e's inequality we can bound $\|u\|_{L^p(\Omega)}^{p-1}\leq c_p(\Omega)^{p-1}\|\nabla u\|_{L^p(\Omega)}^{p-1}$. With the same assumptions of Theorem \ref{teo_n_dim}, it allow us to rewrite  inequality \eqref{ecuteo}  as
$$\left| \int_\Omega (g(\tfrac{x}{\ve})- \bar{g}) |u|^p \right| \leq    C \ve \|\nabla u\|_{L^p(\Omega)}^p,$$
where $C=pc_1 c_p^{p-1} \|g-\bar g\|_{L^\infty(\R^N )}$.

\end{rem}

\begin{proof}[Proof of Theorem \ref{teo_1}:]
 $\lam_1(\bar m)$ can be characterized variationally as
\begin{align} \label{ec44}
\lam_1(\bar m)=\inf_{u\in W_0^{1,p}(\Omega)} \frac{\int_\Omega |\nabla u|^p}{\int_\Omega \bar m |u|^p} =\frac{\int_\Omega |\nabla u_1|^p}{\int_\Omega \bar m |u_1|^p} +o(1)
\end{align}
for some $u_1\in W^{1,p}_0(\Omega)$. We can bound
\begin{align} \label{ec4}
\lam_1(m_\ve)=&\inf_{u\in W_0^{1,p}(\Omega)} \frac{\int_\Omega |\nabla u|^p}{\int_\Omega m_\ve |u|^p} \leq
 \frac{\int_\Omega |\nabla u_1|^p}{\int_\Omega \bar m |u_1|^p}  \frac{\int_\Omega \bar m| u_1|^p}{\int_\Omega m_\ve |u_1|^p}.
\end{align}
By using Theorem \ref{teo_n_dim}, \eqref{cotas} and \eqref{ec44} it follows that
\begin{align} \label{ec5}
\begin{split}
 \frac{\int_\Omega \bar m| u_1|^p}{\int_\Omega m_\ve |u_1|^p} &\leq  1 + c \ve \frac{ \left(\int_\Omega |u_1^p|\right)^{\frac{p-1}{p}}\left(\int_\Omega |\nabla u_1|^p\right)^\frac{1}{p}}{\int_\Omega m_\ve |u_1|^p}\\
 &\leq 1 + c\ve \frac{\bar m^{1/p}}{m_-} \frac{ \left(\int_\Omega |u_1^p|\right)^{\frac{p-1}{p}}\left(\int_\Omega |\nabla u_1|^p\right)^\frac{1}{p}}{\int_\Omega \bar m |u_1|^p}\\
 &\leq 1 + C \ve \left( \frac{\int_\Omega |\nabla u_1|^p}{\int_\Omega \bar m |u_1|^p} \right)^{\frac{1}{p}} \leq 1+C\ve (\lam_1(\bar m)+o(1))^\frac{1}{p},
\end{split}
\end{align}
where $C=pc_1 \|m-\bar m\|_{L^\infty(\R^N)} (m_+)^{1/p}/m_-$.

By replacing \eqref{ec5} and \eqref{ec44},  in \eqref{ec4} we  get
\begin{align} \label{ec7}
\lam_1(m_\ve) - \lam_1(\bar m) \leq C\ve\lam_1(\bar m)^{\frac{1}{p}+1}.
\end{align}
In a similar way, interchanging the roles of $\lam_1(m_\ve)$ and $\lam_1(\bar m)$  we obtain
\begin{align} \label{ec77}
\lam_1(\bar m) - \lam_1(m_\ve) \leq C\ve\lam_1(m_\ve)^{\frac{1}{p}+1}.
\end{align}
By using \eqref{cotas} immediately it follows that
\begin{align} \label{ec777}
\max\{\lam_1(\bar m),\lam_1(m_\ve)\} \leq (m_-)^{-1}\mu_1,
\end{align}
where $\mu_1$ is the first eigenvalue of the Dirichlet $p-$Laplacian.

From equations \eqref{ec7}, \eqref{ec77} and \eqref{ec777} it follows the result.
\end{proof}
In the next Lemma we obtain upper bounds for the coordinates of the first curve of $\Sigma^*(m,n)$.
\begin{lema} \label{lema.2}
Let $m,n$ satisfying $\eqref{cotas}$ and let $(\alpha(s),\beta(s))\in \mathcal C_1(m,n)$. Then for each $s\in\R^+$,
$$\alpha(s) \leq \min\{m_-^{-1},n_-^{-1}\} \mu_2 \tau(s), \qquad \beta(s) \leq \min\{m_-^{-1},n_-^{-1}\} \mu_2 s \tau(s)$$
with $\tau$ defined by
\begin{equation} \label{ftau}
\tau(s)=
\begin{cases}
1 &\quad s \geq 1 \\
s^{-1} &\quad s \leq 1.
\end{cases}
\end{equation}
where $m_-, n_-$ are given by \eqref{cotas} and $\mu_2$ is the second eigenvalue of the $p-$Laplacian equation in $\Omega$ without weights and Dirichlet boundary conditions.
\end{lema}
\begin{proof}
Let $s\in\R^+$. When the parameter $s\geq 1$ we can bound
$$ \lam_1(m) \leq \alpha(s) \leq \alpha(1)=c(m,n). $$

Let $\lam_2(m)$ be the second eigenvalue of the problem \eqref{Plap} with weight $m(x)$. It satisfies that $\alpha(1) \leq \min\{\lam_2(m),\lam_2(n)\}$. By using the assumptions \eqref{cotas} over  $m(x)$, we can bound $\lam_2(m)$ by $ \mu_2 m_-^{-1}$, where $\mu_2$ is the second eigenvalue of the $p-$Laplacian equation with Dirichlet boundary conditions on $\Omega$. Analogously for $\lam_2(n)$. We get
\begin{equation} \label{eq12}
  \alpha (s) \leq \alpha (1) \leq  \min\{m_-^{-1},n_-^{-1}\}\mu_2 , \quad s\geq 1
\end{equation}
\begin{figure}[ht]
\begin{center}
\includegraphics[width=7.58cm,height=6.0cm]{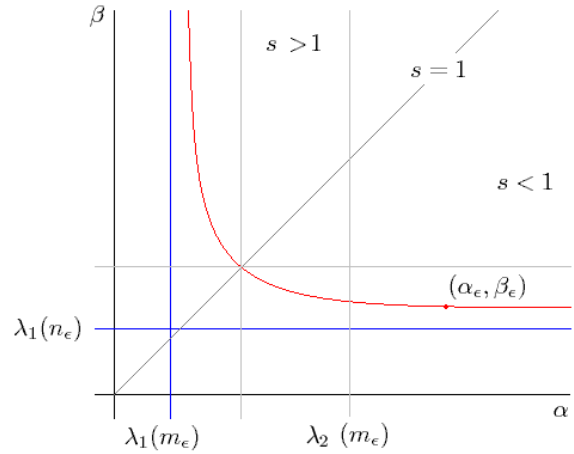}
\caption{The first curve of the spectrum.}
\end{center}
\end{figure}

When $s\leq 1$ the following bound holds for the second coordinate of $\mathcal{C}_\ve$
\begin{equation} \label{eq122}
\lam_1(n ) \leq \beta (s)\leq \beta (1).
\end{equation}
By multiplying \eqref{eq122} by $s^{-1}$ and by using that $\beta (s)=s\alpha (s)$  we have
$$ s^{-1}\lam_1(n ) \leq  \alpha (s) \leq  s^{-1} \beta (1). $$
Being $\alpha(1)=\beta(1)$, it follows that
\begin{equation} \label{eq13}
 \alpha (s) \leq  s^{-1} \alpha (1) \leq s^{-1}\min\{m_-^{-1},n_-^{-1}\}\mu_2,\quad s \leq 1.
\end{equation}
By using \eqref{eq12}, \eqref{eq13} and the relation $\beta(s)=s\alpha(s)$ the conclusions of the lemma follows.

\end{proof}

The following Proposition gives the monotonicity of $c(\cdot,\cdot)$:
\begin{prop}[Proposition 23, \cite{ARI}] \label{propo}
If $m\leq \tilde{m}$ and $n\leq \tilde{n}$ a.e., then
$$c(\tilde{m},\tilde{n}) \leq c(m,n),$$
where $c(\cdot,\cdot)$ is defined by \eqref{defc}.
\end{prop}

In the next Lemma we obtain lower bounds for the coordinates of the first curve of $\Sigma^*(m,n)$.

\begin{lema} \label{lema.3}
Let $m,n$ satisfying $\eqref{cotas}$ and let $(\alpha(s),\beta(s))\in \mathcal C(m,n)$. Then for each $s\in\R^+$,
$$\alpha(s) \geq \tfrac{1}{s} C \omega(s) , \qquad \beta(s) \geq C\omega(s) $$
with $\omega$ defined by
\begin{equation} \label{fomega}
\omega(s)=
\begin{cases}
1 &\quad s \geq 1 \\
s &\quad s \leq 1
\end{cases}
\end{equation}
where $C$ is a positive constant depending only of the bounds given in \eqref{cotas}.
\end{lema}

\begin{proof}
Let $s\in\R^+$. When the parameter $s\geq 1$ we can bound bellow
$$ \beta(s) \geq \beta(1)=c(m,n), \quad  s\geq 1. $$
Using the relation $\beta(s)=s\alpha(s)$ we obtain
$$ \alpha(s) \geq s^{-1}c(m,n), \quad s\geq 1. $$
Similarly, when $s\leq 1$ we have
$$ \alpha(s) \geq \alpha(1)=c(m,n), \quad  s\leq 1, $$
and again, by the relation between $\alpha(s)$ and $\beta(s)$ we get
$$ \beta(s) \geq s c(m,n), \quad  s\leq 1, $$
Using   \eqref{cotas}   and Proposition \ref{propo}, we can bound bellow
$$c(m,n)\geq c(m_+,n_+).$$
and the result follows.
\end{proof}

Now we are able to prove Theorem \ref{teo_2}.
\begin{proof}[Proof of Theorem \ref{teo_2}:]
For each fixed value of $\ve>0$, by \eqref{cotas} together with the monotonicity of $c(\cdot,\cdot)$ provided by Proposition \ref{propo}, we can assert that there exist two curves $\mathcal C_1^+(m_+,n_+)$ and $\mathcal C_1^-(m_-,n_-)$ such that   delimit above and below to the curve $\mathcal{C}_1^\ve(m_\ve,n_\ve)$.  It follows that for each fixed value of $s$, $\alpha_\ve(s)$ and $\beta_\ve(s)$ are bounded.

Let $(\alpha_\ve,\beta_\ve)$ be a point belonging to the curve $\mathcal{C}_1^\ve(m_\ve,n_\ve)$ and let $(\alpha_0,\beta_0)$ be the point obtained when $\ve \cf 0$. Let us see that it belongs to $\mathcal{C}_1(\bar m,\bar n)$.

Fixed a value of $\ve>0$  and by using  \eqref{defc}, the inverse of $c(m_\ve,n_\ve)$ can be written as
\begin{equation} \label{inve}
 \frac{1}{c(m_\ve,n_\ve)}=\sup_{\gamma \in \Gamma } \inf_{u\in \gamma[-1,+1]}  B_{m_\ve,n_\ve}(u)
\end{equation}
where
$$\Gamma =\{\gamma \in C(I,H):\gamma(-1)\geq 0 \textrm{ and } \gamma(1) \leq 0\}$$
for $I:=[-1,+1]$ and
$$H =\{u\in W^{1,p}_0(\Omega):A(u)=1\}$$
$A$ and $B$ being the functionals defined in \eqref{funcionales}.

By \eqref{caract} and \eqref{inve} we have the following characterization for the inverse of $\alpha_\ve(s)$
\begin{equation} \label{eq1}
 \frac{1}{\alpha_\ve(s)} =\frac{1}{c(m_\ve,s n_\ve)}=\sup_{\gamma \in \Gamma } \inf_{u\in \gamma(I)}  B_{m_\ve,s n_\ve}(u).
\end{equation}
Similarly, we can consider an equation analog to \eqref{eq1} for the representation of the inverse of $\alpha_0(s)$.
Let $\delta>0$ and $\gamma_1(\delta) \in \Gamma$ such that
\begin{equation} \label{eq2}
\frac{1}{\alpha_0(s)} = \inf_{u\in \gamma_1(I)}   B_{\bar m,s \bar n}(u)    + O(\delta).
\end{equation}
In order to find a bound for $a_\ve$ we use $\gamma_1 \in \Gamma_1$, which is admissible in its variational characterization,
\begin{align} \label{eq3}
\frac{1}{\alpha_\ve(s)} \geq \inf_{u\in \gamma_1(I)}  B_{m_\ve,s n_\ve}(u)  .
\end{align}
As $u\in W^{1,p}_0(\Omega)$, it follows that $(u^+)^p$ and $(u^-)^p$ belong to $W^{1,1}_0(\Omega)$.
This allows us to estimate the error by replacing the oscillating weights by their averages by using Remark \ref{teo_n_dim_rem}. For each  fixed function $u\in \gamma_1(I)$ we bound
\begin{align} \label{eq4}
B_{m_\ve,s n_\ve}(u) \geq  B_{\bar m,s \bar n}(u)  - c_{ m} \ve \|\nabla u^+\|^p_{L^p(\Omega)} - c_{n} \ve s \|\nabla u^-\|^p_{L^p(\Omega)}
\end{align}
where $c_{ m}$ and $c_{ m}$ are the constants  given in Remark \ref{teo_n_dim_rem}.
As  $u\in H$ we have
\begin{align} \label{eq5}
 \|\nabla u^+\|_{L^p(\Omega)}^p \leq 1, \quad  \|\nabla u^-\|_{L^p(\Omega)}^p \leq 1.
\end{align}
So, from  \eqref{eq5} and \eqref{eq4}, taking $c=\max\{c_{ m}, c_{ n}\}$ we get
\begin{align} \label{eq6}
B_{m_\ve,s n_\ve}(u) \geq  B_{\bar m ,s \bar n }(u)  -c\ve(1+s).
\end{align}
Taking infimum over the functions $u$ in $\gamma_1(I)$  together with \eqref{eq2} and \eqref{eq3} we obtain
\begin{align*}
\alpha_\ve^{-1}(s)-\alpha_0^{-1}(s) \geq - c\ve(1+s) +O(\delta).
\end{align*}
Letting $\delta\cf 0$ we get
\begin{align} \label{eq9}
\alpha_\ve^{-1}(s)-\alpha_0^{-1}(s) \geq - c\ve(1+s).
\end{align}
In a similar way, interchanging the roles of $\alpha_\ve$ and $\alpha_0$  we obtain the inequality
\begin{align} \label{eq10}
\alpha_\ve^{-1}(s)-\alpha_0^{-1}(s) \leq  c\ve(1+s).
\end{align}
From equations \eqref{eq9} and \eqref{eq10} it follows that
\begin{equation} \label{eq11}
|\alpha_\ve(s) - \alpha_0(s)| \leq c \ve (1+s) \alpha_\ve(s) \alpha_0(s).
\end{equation}
By using Lemma \ref{lema.2} we can bound the expression \eqref{eq11} as
\begin{equation*}
|\alpha_\ve(s) - \alpha_0(s)| \leq c(\min\{m_-^{-1},n_-^{-1}\}\mu_2)^2 (1+s) \tau(s)^2 \ve.
\end{equation*}
where $\tau(s)$ is given by \eqref{ftau} and $\mu_2$ is the second eigenvalue of the Dirichlet $p-$Laplacian.

From the convergence of $\alpha_\ve$ together with \eqref{caract} it follows the convergence of $\beta_\ve$ and of the whole curve.
\end{proof}

The proof of Theorem \ref{teo_2.gral}, where general weights are considered, is analogous to that of Theorem \ref{teo_2} but we need a result similar to Theorem \ref{teo_n_dim} that works without assuming periodicity. It is found in the following theorem.
\begin{thm} \label{teo_n_dim.gral}
Let $\Omega\subset \R^N$ be a bounded domain. Let $g_\ve$ be a function such that   $0<g^- \leq g_\ve \leq g^+<+\infty$ for $g^\pm$ constants and $g_\ve \cd g$ weakly* in $L^\infty(\Omega)$. Then for every $u\in W^{1,p}(\Omega)$,
$$\lim_{\ve\cf 0} \int_\Omega (g_\ve- g) |u|^p =0$$
where $1< p < +\infty$.
\end{thm}
\begin{proof}
The weak* convergence of $g_\ve$ in $L^\infty(\Omega)$ says that  $\int_\Omega g_\ve \varphi \cf \int_\Omega g \varphi$ for all $\varphi \in L^1(\Omega)$. Particularly, $u\in W^{1,p}(\Omega)$ implies that $|u|^p\in W^{1,1}(\Omega)$, it follows that $|u|^p\in L^1(\Omega)$ and the result is proved.
\end{proof}

\begin{proof}[Proof of Theorem \ref{teo_2.gral}:]
The argument follows exactly as in the proof of Theorem \ref{teo_2} using the Theorem \ref{teo_n_dim.gral} instead of the Theorem \ref{teo_n_dim}.
\end{proof}

\section{Neumann boundary conditions} \label{sec.neu}
Let $\Omega$ be a bounded domain in $\R^N$, $N\geq 1$ with Lipschitz boundary and let $m,n$ be two weights satisfying \eqref{cotas}. We consider the following asymmetric elliptic problem with Neumann boundary conditions
\begin{align} \label{P2.neu}
\begin{cases}
-\Delta_p u=\alpha m(x)(u^+ )^{p-1} -\beta n(x)(u^- )^{p-1} \quad &\textrm{ in } \Omega \\
\tfrac{\partial u}{\partial \eta }=0 \quad &\textrm{ on } \partial \Omega.
\end{cases}
\end{align}
where $\eta$ denotes the unit exterior normal.

Let $r(x)$ be a weight satisfying \eqref{cotas}. Now, $0$ is a principal eigenvalue of
\begin{align} \label{Plap.neu}
\begin{cases}
-\Delta_p u= \lam r(x) |u|^{p-2}u  \quad &\textrm{ in } \Omega \\
\tfrac{\partial u}{\partial \nu}=0 \quad &\textrm{ on } \partial \Omega
\end{cases}
\end{align}
with constants as eigenfunctions. Moreover, the positivity of $r$ guaranties that $0$ is the unique nonnegative principal eigenvalue, see \cite{GGP}.

Consequently, the Fu\u{c}ik spectrum $\Sigma=\Sigma(m,n)$ clearly contains the lines $\{0\}\times \R$ and  $\R \times \{0\}$. We denote by $\Sigma^*=\Sigma^*(m,n)$ the set $\Sigma(m,n)$ without these two lines.

In this case, when $N>1$ only a full description of the first nontrivial curve of $\Sigma$, which we will denote by $\mathcal{C}_1=\mathcal{C}_1(m,n)$. Moreover, in \cite{AR3} (see Theorem 6.1) a  characterization similar to the Dirichlet case is given:
\begin{equation} \label{curva.neu}
 \mathcal{C}_1=\{(\alpha(s),\beta(s)) , s\in\R^+\}
\end{equation}
where $\alpha(s)$ and $\beta(s)$ are continuous functions defined by
$\alpha(s)=c(m,sn)$, $\beta(s)=s\alpha(s)$ and $c(\cdot,\cdot)$ is given by
\begin{equation} \label{defc.n}
 c(m,n)=\inf_{\gamma \in \Gamma} \max_{u\in \gamma(J)} \frac{A(u)}{B(u)}.
\end{equation}
with $J:=[0,1]$, the functionals $A$ and $B$ given by \eqref{funcionales}, and
$$\Gamma =\{\gamma \in C(J,W^{1,p}(\Omega)):\gamma(0)\geq 0 \textrm{ and } \gamma(1) \leq 0\}.$$

Let  $m_\ve$ and $n_\ve$ be two functions such that satisfy \eqref{cotas} and \eqref{limi}. We consider the following problem depending on $\ve>0$
\begin{align} \label{P1.neu}
\begin{cases}
-\Delta_p u_\ve=\alpha_\ve m_\ve(u_\ve^+ )^{p-1} - \beta_\ve n_\ve(u_\ve^- )^{p-1} \quad &\textrm{ in } \Omega \\
\tfrac{\partial u_\ve}{\partial \nu}=0 \quad &\textrm{ on } \partial \Omega.
\end{cases}
\end{align}

 As we made with the Dirichlet equation \eqref{P1}, we want to study the behavior of the first non-trivial curve in the spectrum of \eqref{P1.neu} as $\ve\cf0$. When $\ve$ tends to zero in \eqref{P1.neu}, according to \eqref{limi} we obtain the following limit equation
\begin{align} \label{pron1limite.gral.neu}
\begin{cases}
-\Delta_p u_0=\alpha_0 m_0(x) (u_0^+ )^{p-1} - \beta_0 n_0(x)(u_0^- )^{p-1} \quad &\textrm{ in } \Omega \\
\tfrac{\partial u_0}{\partial \nu}=0 \quad &\textrm{ on } \partial \Omega.
\end{cases}
\end{align}

Analogously to Theorem \ref{teo_2.gral}, we obtain the following result of convergence:

\begin{thm} \label{teo_2.gral.neu}
Let $m_\ve,n_\ve$ satisfying \eqref{cotas}, and \eqref{limi}. Then the first non-trivial curve of problem \eqref{P1.neu}
$$\mathcal{C}^\ve_1:=\mathcal{C}_1(m_\ve,n_\ve)=\{\alpha_\ve(s),\beta_\ve(s), s\in\R^+\}$$
converges to the first non-trivial curve  of the limit problem \eqref{pron1limite.gral.neu}
$$\mathcal{C}_1:=\mathcal{C}_1(m_0,n_0)=\{\alpha_0(s),\beta_0(s), s\in\R^+\}$$
as $\ve \cf 0$ in the sense that $\alpha_\ve(s)\cf \alpha_0(s)$, $\beta_\ve(s)\cf \beta_0(s)$ $\forall s\in\R^+$.
\end{thm}

When the case of periodic homogenization is considered, i.e.,  $m_\ve(x)=m(x/\ve)$ and $n_\ve(x)=n(x/\ve)$ with $m$ and $n$  $Q-$periodic functions, $Q$ being the unit cube in $\R^N$, the limit functions $m_0$, $n_0$ given in \eqref{limi} are equal to the averages of $m$ and $n$ over $Q$, respectively, i.e., $m_0=\bar m$ and $n_0=\bar n$. Now, like in the Dirichlet case, in addition to the convergence of the first non-trivial curve, we obtain the  convergence rates:

\begin{thm} \label{teo_2.neu}
Let $\Omega\subset\R^N$, $N>1$ be a bounded domain with $C^1$ boundary. Under the same considerations of Theorem \ref{teo_2.gral.neu}, if the weights $m_\ve$ and $n_\ve$ are given in terms of $Q-$periodic functions $m,n$ in the form $m_\ve(x)=m(\tfrac{x}{\ve})$ and $n_\ve(x)=n(\tfrac{x}{\ve})$, for each $s\in \R^+$ we have the following estimate
\begin{align*}
|\alpha_\ve(s) - \alpha_0(s)| \leq c  (1+s) \tau(s) \ve   , \quad |\beta_\ve(s) - \beta_0(s)| \leq c s(1+s) \tau(s) \ve
\end{align*}
where $c=c(\Omega,p,m,n)$ is a constant independent of $\ve$ and $s$, and $\tau$ is given by \eqref{ftau.2}.
\end{thm}

To prove Theorem \ref{teo_2.neu} arguments used in the Dirichlet case fail. This is due to the fact that now the functions space is $W^{1,p}(\Omega)$ but Theorem \ref{teo_n_dim} holds only for functions in $W^{1,p}_0(\Omega)$. The fact of enlarge the set of test functions is reflected in the need for more regularity of the domain $\Omega$. We will prove the following result which works with functions belonging to $W^{1,p}(\Omega)$.

\begin{thm} \label{teo.neu}
Let $\Omega\subset \R^N$ be a bounded domain with $C^1$ boundary and denote by $Q$ to the unit cube in $\R^N$. Let $g$ be a $Q-$periodic bounded function. Then for every $u\in W^{1,p}(\Omega)$ there exists a constant $c$ independent of $\ve$ such that
$$
\left| \int_\Omega (g(\tfrac{x}{\ve})  - \bar{g}) u \right| \leq c\ve \|u\|_{W^{1,p}(\Omega)}
$$
where $\bar{g}=\int_Q g $ and $1\leq p < +\infty$.
\end{thm}

\begin{rem}
Unlike to Theorem \ref{teo_n_dim},  we are not able to compute explicitly the constant $c$ in in Theorem \ref{teo.neu}.
\end{rem}
\section{Proof of the Neumann results}
We begin this section by proving some auxiliary results that are essential to prove Theorem \ref{teo.neu}. The next lemma is a generalization for $p\geq 2$ of Oleinik's Lemma \cite{Ol}.

\begin{lema}\label{lemaaux}
Let $\Omega\subset \R^N$  be a bounded domain with $C^1$ boundary and, for $\delta>0$,  let $G_\delta$ be a tubular neighborhood of $\partial\Omega$, i.e. $G_\delta=\{x\in \Omega \colon dist(x,\partial \Omega)<\delta\}$.
Then there exists $\delta_0>0$ such that for every $\delta\in (0,\delta_0)$ and every $v\in W^{1,p}(\Omega)$ we have
$$
\|v\|_{L^p(G_\delta)}\leq c\delta^{\frac{1}{p}}\|v\|_{W^{1,p}(\Omega)},
$$
where $c$ is a constant independent of $\delta$ and $v$.
\end{lema}

\begin{proof}
Let $G_\delta=\{x\in \Omega \colon dist(x,\partial \Omega)<\delta\}$, it follows that $S_\delta=\partial G_\delta$ are uniformly smooth surfaces.

By the Sobolev trace Theorem we have
$$
\|v\|_{L^p(S_\delta)}^p = \int_{S_\delta}|v|^p dS \leq c  \|v\|^p_{W^{1,p}(\Omega_\delta)}\leq c  \|v\|^p_{W^{1,p}(\Omega)} \quad \delta\in[0,\delta_0],
$$
where $c $ is a constant independent of $\delta$. Integrating this inequality with respect to $\delta$ we get
$$\|v\|^p_{L^p(G_\delta)}=\int_0^\delta \Big(\int_{S_\tau} |v|^p dS \Big)d\tau \leq c  \delta \|v\|^p_{W^{1,p}(\Omega)}
$$
and the Lemma is proved.
\end{proof}

The next Theorem is essential to estimate the rate of convergence of the eigenvalues since it allows us to replace an integral involving a rapidly oscillating function with one that involves its average in the unit cube.
First, we need an easy Lemma that computes the Poincar\'e constant on the cube of side $\ve$ in terms of the Poincar\'e constant of the unit cube. Although this result is well known and its proof follows directly by a change of variables, we choose to include it for the sake of completeness.

\begin{lema}\label{poincare}
Let $Q$ be the unit cube in $\R^N$ and let $c_p$ be the Poincar\'e constant in the unit cube in $L^p$, i.e.
$$
\| u - (u)_{Q}\|_{L^p(Q)}\le c_p \| \nabla u\|_{L^p(Q)}, \qquad \mbox{for every } u\in W^{1,p}(Q),
$$
where $(u)_{Q}$ is the average of $u$ in $Q$. Then, for every $u\in W^{1,p}(Q_\ve)$ we have
$$
\| u - (u)_{Q_\ve}\|_{L^p(Q_\ve)}\le c_p \ve \| \nabla u\|_{L^p(Q_\ve)},
$$
where $Q_\ve = \ve Q$.
\end{lema}

\begin{proof}
Let $u\in W^{1,p}(Q_\ve)$. We can assume that $(u)_{Q_\ve}=0$. Now, if we denote $u_\ve (y) = u(\ve y)$, we have that $u_\ve\in W^{1,p}(Q)$ and by the change of variables formula, we get
\begin{align*}
\int_{Q_\ve} |u|^p &= \int_{Q} |u_\ve|^p \ve^n  \le c_p^p \ve^n \int_{Q} |\nabla u_\ve|^p = c_p^p \ve^p \int_{Q_\ve} |\nabla u|^p.
\end{align*}
The proof is now complete.
\end{proof}

\begin{thm} \label{teo_n_dim_v2}
Let $\Omega\subset \R^N$ be a bounded domain with smooth boundary and denote by $Q$ the unit cube in $\R^N$. Let $g$ be a $Q$-periodic bounded function such that $(g)_Q=0$. Then the inequality
$$\left| \int_\Omega g\left(\frac{x}{\ve}\right)uv \right| \leq c\ve \|u\|_{W^{1,p}(\Omega)}\|v\|_{W^{1,p'}(\Omega)}$$
holds for every $u\in W^{1,p}(\Omega)$ and $v\in W^{1,p'}(\Omega)$, where $c$ is a
constant independent of $\ve$, $u$, $v$ and $p,\ p'$ are conjugate exponents.
\end{thm}

\begin{proof}
Denote by $I^\ve$ the set of all $z\in \Z^N$ such that $Q_{z,\ve}:=\ve(z+Q)\subset\Omega$. Set $\Omega_1=\bigcup_{z\in I^\ve} Q_{z,\ve}$ and $G=\Omega\setminus\bar{\Omega}_1$. Let us consider the functions $\bar{v}$ and $\bar{u}$ given by the formulas
$$
\bar{v}(x)=\frac{1}{\ve^n}\int_{Q_{z,\ve}} v(x)dx,\quad \bar{u}(x)=\frac{1}{\ve^n}\int_{Q_{z,\ve}} u(x)dx
$$
for $x\in Q_{z,\ve}$. Then we have
\begin{align} \label{keq0}
\begin{split}
\int_\Omega g_\ve uv &= \int_G g_\ve uv + \int_{\Omega_1} g_\ve uv\\
&=\int_G g_\ve uv +\int_{\Omega_1} g_\ve (u-\bar{u})v  + \int_{\Omega_1} g_\ve \bar{u}(v-\bar{v}) + \int_{\Omega_1} g_\ve \bar{v} \bar{u}.
\end{split}
\end{align}
The set $G$ is a $\delta$-neighborhood of $\partial \Omega$ with $\delta = c\ve$ for some constant $c$, and therefore according to Lemma \ref{lemaaux} we have
\begin{align} \label{keq3}
\begin{split}
\|u\|_{L^p(G)}\leq c \ve^{\frac{1}{p}}\|u\|_{W^{1,p}(\Omega)} \\
\|v\|_{L^{p'}(G)}\leq c \ve^{\frac{1}{p'}}\|v\|_{W^{1,p'}(\Omega)}.
\end{split}
\end{align}
As $g$ is bounded, we get
\begin{align}\label{cotan1}
\int_G g_\ve uv &\leq c \|u\|_{L^p(G)} \|v\|_{L^{p'}(G)} \leq c \ve  \|u\|_{W^{1,p}(\Omega)} \|v\|_{W^{1,p'}(\Omega)}.
\end{align}
Now, by Lema \ref{poincare} we get
\begin{equation}\label{5.4}
\begin{split}
\|u-\bar{u}\|_{L^p(\Omega_1)} &= \left( \sum_{z\in I^\ve} \int_{Q_{z,\ve}} |u-\bar{u}|^p dx\right)^{\frac{1}{p}}
\leq c_p\ve \left(\sum_{z\in I^{z,\ve}} \int_{Q_{z,\ve}} |\nabla u(x)|^p dx\right)^{\frac{1}{p}} \\
& = c_p\ve \|\nabla u\|_{L^p(\Omega_1)}.
\end{split}
\end{equation}
Analogously
\begin{align} \label{cot1}
\|v-\bar{v}\|_{L^{p'}(\Omega_1)}\leq c_{p'} \ve \|\nabla v\|_{L^{p'}(\Omega_1)}
\end{align}
By the definition of $\bar{u}(x)$ we get
\begin{align} \label{cot2}
\begin{split}
\|\bar{u}\|_{L^p(\Omega_1)}^p  &= \sum_{z\in I^\ve} \int_{Q_{z,\ve}} |\bar u|^p = \sum_{z\in I^\ve} \ve^n \Big(\ve^{-n}\int_{Q_{z,\ve}} u\Big)^p\\
&\le \ve^{n-np} \sum_{z\in I^\ve} |Q_{z,\ve}|^{p/p'}\int_{Q_{z,\ve}} |u|^p = \ve^{n-np+np/p'} \sum_{z\in I^\ve} \int_{Q_{z,\ve}} |u|^p\\
& = \int_{\Omega_1} |u|^p = \| u\|_{L^p(\Omega_1)}^p.
\end{split}
\end{align}
Finally, since $(g)_{Q_1}=0$ and since $g$ is $Q-$periodic, we get
\begin{equation}\label{ultima}
\int_{\Omega_1} g_\ve \bar u \bar v = \sum_{z\in I^\ve} \bar u \bar v \int_{Q_{z,\ve}} g_\ve = 0.
\end{equation}

Now, combining \eqref{cotan1}, \eqref{5.4}, \eqref{cot1}, \eqref{cot2} and \eqref{ultima} we can bound \eqref{keq0} by
$$
\int_\Omega g_\ve uv\le C \ve \|u\|_{W^{1,p}(\Omega)} \|v\|_{W^{1,p'}(\Omega)}.
$$
This finishes the proof.
\end{proof}

Now we are ready to proof Theorem \ref{teo.neu}:

\begin{proof}[Proof of Theorem \ref{teo.neu}:]
The result follows applying Theorem \ref{teo_n_dim_v2} to
$\widetilde{g}_\epsilon=g_\epsilon-\bar{g}$ and
taking $v\equiv 1$.
\end{proof}

\bigskip

\begin{rem} \label{remar}
Let us observe that $u\in W^{1,p}(\Omega)$ is solution of  equation \eqref{P2.neu} if and only if $u$ is solution of  equation
\begin{equation} \label{P22.neu}
-\Delta_p u +m(u^+)^{p-1}+n(u^-)^{p-1}=\tilde \alpha m(u^+ )^{p-1} -\tilde \beta n(u^- )^{p-1} \quad \textrm{ in } \Omega.
\end{equation}
with Neumann boundary conditions,
where $\tilde \alpha=\alpha-1$ and $\tilde \beta=\beta+1$.
The main advantage between consider equations \eqref{P2.neu} and \eqref{P22.neu} is the fact that in the second one the functional $A(u)$ defined in \eqref{funcionales} becomes in
\begin{equation} \label{ecu.n}
  A_{m,n}(u)=\int_\Omega |\nabla u|^p + m(u^+)^p + n(u^-)^p dx,
\end{equation}
which involves both $\nabla u$ and the function $u$.
\end{rem}

\medskip

\begin{proof}[Proof of Theorem \ref{teo_2.neu}:]
The proof is similar to that of Theorem \ref{teo_2} for the Dirichlet case. According to Remark \ref{remar} we consider equation \eqref{P22.neu}.
Let $(\tilde \alpha_\ve,\tilde \beta_\ve)$ be a point belonging to the curve $\mathcal{C}_1^\ve (m_\ve,n_\ve)$ and let $(\tilde \alpha_0,\tilde \beta_0)$ be the point obtained when $\ve \cf 0$. It follows that $(\tilde \alpha_0,\tilde \beta_0)$ belongs to the spectrum of the limit equation. Let us see that it belongs to $\mathcal{C} (\bar m,\bar n)$.
The main difference is that in the characterization \eqref{defc.n} of $c(m_\ve,n_\ve)$, now  we are considering
$$  \Gamma =\{\gamma \in C(J, W^{1,p}(\Omega)):\gamma(0)\geq 0 \textrm{ and } \gamma(1) \leq 0\}.$$
with $J:=[0,1]$.
Fixed a value of $\ve>0$  we write
\begin{equation} \label{inve.n}
 c(m_\ve,n_\ve)=\inf_{\gamma \in  \Gamma} \sup_{u\in \gamma }  \frac{A_{m_\ve,n_\ve}(u)}{B_{m_\ve,n_\ve}(u)}.
\end{equation}
By \eqref{caract} and \eqref{inve.n} we have the following characterization   of $\tilde\alpha_\ve(s)$
\begin{equation} \label{eq1.n}
  \tilde \alpha_\ve(s)  = c(m_\ve,s n_\ve) =\inf_{\gamma \in   \Gamma} \sup_{u\in \gamma }  \frac{A_{m_\ve,n_\ve}(u)}{B_{m_\ve,s n_\ve}(u)}.
\end{equation}
Similarly, we can consider an equation analog to \eqref{eq1.n} for the representation of $\tilde\alpha_0(s)$.
Let $\delta>0$ and $\gamma_1=\gamma_1(\delta) \in   \Gamma$ such that
\begin{equation} \label{eq2.n}
\tilde \alpha_0(s) = \sup_{u\in \gamma_1 }   \frac{{A_{\bar m,\bar n}(u)}}{B_{\bar m ,s \bar n}(u)}    + O(\delta).
\end{equation}
In order to find a bound for $\tilde a_\ve$ we use $\gamma_1 \in \Gamma$, which is admissible in its variational characterization,
\begin{align} \label{eq3.n}
\tilde\alpha_\ve(s) \leq \sup_{u\in \gamma_1 }  \frac{A_{m_\ve,s n_\ve}(u)}{B_{\bar m ,s \bar n }(u)} \frac{B_{\bar m,s \bar n}(u)}{B_{m_\ve,s n_\ve}(u)} .
\end{align}
To bound $\tilde\alpha_\ve$ we look for bounds of the two quotients in \eqref{eq3.n}. Since $u\in W^{1,p}(\Omega)$, by Theorem \ref{teo.neu} we obtain that
$$
\frac{A_{m_\ve, n_\ve}(u)}{B_{\bar m ,s \bar n }(u)}  \leq \frac{A_{\bar m , \bar n }(u)}{B_{\bar m ,s \bar n }(u)}+\frac{c\ve\||u^+|^p\|_{W^{1,1}(\Omega)}+c\ve\||u^-|^p\|_{W^{1,1}(\Omega)}}{B_{\bar m ,s \bar n}(u)}.
$$
For every function $u\in \gamma_1$ we have that
\begin{equation} \label{exu1}
\frac{A_{\bar  m, \bar n}(u)}{B_{\bar m,s \bar n}(u)} \leq \sup_{u\in\gamma_1} \frac{A_{\bar m , \bar n }(u)}{B_{\bar m ,s \bar n }(u)} = \tilde \alpha_0(s)+O(\delta).
\end{equation}
By using Young  inequality, for each $v\in W^{1,p}(\Omega)$
\begin{align} \label{ecyo}
\begin{split}
 \||v|^p\|_{W^{1,1} (\Omega)} &= \||v|^p\|_{L^1(\Omega)} + p \| |v|^{p-1} \nabla v\|_{L^1(\Omega)}   \\
 &= \|v\|_{L^p(\Omega)}^p + p \| |v|^{p-1} \nabla v\|_{L^1(\Omega)}   \\
 &\leq p\|v\|_{L^p(\Omega)}^p + \|\nabla v\|_{L^p(\Omega)}^p  .
\end{split}
\end{align}
From \eqref{ecyo} it follows that
\begin{align} \label{exu2}
\begin{split}
\frac{\||u^+|^p\|_{W^{1,1}(\Omega)}}{B_{\bar m ,s \bar  n }(u)} &\leq \frac{p\|u^+\|_{L^p(\Omega)}^p + \|\nabla u^+\|_{L^p(\Omega)}^p }{B_{\bar m ,s \bar n }(u)}\\
&\leq c \frac{A_{\bar m ,\bar n }(u)}{B_{\bar m ,s \bar n }(u)}  \\
&\leq c\sup_{u\in\gamma_1} \frac{A_{\bar m ,\bar n}(u)}{B_{\bar m ,s \bar n }(u)}  \\
&=c (\tilde\alpha_0(s)+O(\delta)),
\end{split}
\end{align}
and similarly
\begin{equation} \label{exu3}
\frac{\||u^-|^p\|_{W^{1,1}(\Omega)}}{B_{\bar m,s \bar n}(u)} \leq c (\tilde\alpha_0(s)+O(\delta)).
\end{equation}
To bound the second quotient in \eqref{eq3.n}, we use again Theorem \ref{teo.neu} and \eqref{cotas} to obtain
\begin{align} \label{exu4}
\begin{split}
\frac{\int_\Omega \bar m  |u^+|^p}{B_{m_\ve,s n_\ve}(u)} &\leq
\frac{\int_\Omega m_\ve |u^+|^p}{B_{m_\ve,s n_\ve}(u)} +c \ve \frac{\||u^+|^p\|_{W^{1,1}(\Omega)}}{B_{m_\ve,s n_\ve}(u)}\\
&\leq \frac{\int_\Omega m_\ve |u^+|^p}{B_{m_\ve,s n_\ve}(u)}+ c\ve \frac{\||u^+|^p\|_{W^{1,1}(\Omega)}}{B_{\bar m,s \bar n}(u)},
\end{split}
\end{align}
and similarly
 \begin{align} \label{exu5}
\frac{\int_\Omega s \bar n |u^-|^p}{B_{m_\ve,s n_\ve}(u)} \leq
 \frac{\int_\Omega s n_\ve |u^+|^p}{B_{m_\ve,s n_\ve}(u)}+ s c\ve \frac{\||u^-|^p\|_{W^{1,1}(\Omega)}}{B_{\bar m ,s \bar n }(u)}.
\end{align}
Now, from equations \eqref{exu4},\eqref{exu5} together with \eqref{exu2} and \eqref{exu3} we get
\begin{align}\label{exu6}
\begin{split}
\frac{B_{\bar m ,s \bar n }(u)}{B_{m_\ve,s n_\ve}(u)} &=
\frac{\int_\Omega \bar m |u^+|^p+\int_\Omega s \bar n |u^-|^p}{B_{m_\ve,s n_\ve}(u)} \\
&\leq 1+(1+s)c\ve (\tilde\alpha_0(s)+O(\delta)).
\end{split}
\end{align}
Then combining \eqref{eq3.n},\eqref{exu2},\eqref{exu3} and  \eqref{exu6} we find that
$$
\tilde \alpha_\ve(s) \leq \left((\tilde \alpha_0(s)+O(\delta))+c\ve(\tilde \alpha_0(s)+O(\delta)) \right)\left( 1+(1+s)c\ve (\tilde\alpha_0(s)+O(\delta))  \right).
$$
Letting $\delta\cf 0$ we get
\begin{equation} \label{cotax1}
\tilde \alpha_\ve(s) -\tilde \alpha_0(s)\leq  c\ve(\tilde\alpha_0^2 (1+s)+\tilde\alpha_0).
\end{equation}
In a similar way, interchanging the roles of $\tilde\alpha_0$ and $\tilde\alpha_\ve$, we obtain
\begin{equation} \label{cotax2}
\tilde \alpha_0(s)-\tilde \alpha_\ve(s)\leq  c\ve(\tilde\alpha_\ve^2 (1+s)+\tilde\alpha_\ve).
\end{equation}
From \eqref{cotax1} and \eqref{cotax2} we arrive at
$$|\tilde \alpha_0(s)-\tilde \alpha_\ve(s) | \leq c\ve (1+s)\max\{\tilde \alpha_0(s)^2,\tilde \alpha_\ve(s)^2\}.$$
Now, using Lemma \ref{lema.2},
\begin{align*}
|\alpha_\ve(s) - \alpha_0(s)|\leq c(1+s ) \tau(s)^2 \ve,
\end{align*}
where $c$ is a constant independent of $\ve$ and $s$, and $\tau(s)$ is given by \eqref{ftau}. Here, Lemma \ref{lema.2} holds in the Neumann case, but now we have
 $$\alpha(s) \leq \min\{m_-^{-1},n_-^{-1}\} \mu_2 \tau(s), \qquad \beta(s) \leq \min\{m_-^{-1},n_-^{-1}\} \mu_2 s \tau(s)$$
 with $\mu_2$ the second eigenvalue of the $p-$Laplacian equation on $\Omega$ with Neumann boundary conditions.
From the convergence of $ \alpha_\ve$ and \eqref{caract} it follows the convergence of $ \beta_\ve$ and of the whole curve.
\end{proof}

\begin{proof}[Proof of Theorem \ref{teo_2.gral.neu}:]
As Theorem \ref{teo_n_dim.gral} holds for functions belonging to $W^{1,p}(\Omega)$ with $\Omega$ any bounded domain in $\R^N$, this proof is analogous to those of Theorem \ref{teo_2.gral}.
\end{proof}

 \section*{Acknowledgements}
I want to thank Juli\'an Fern\'andez Bonder for valuable
help and many stimulating discussions on the subject of the paper.

\bibliographystyle{amsplain}

\begin{thebibliography}{10}

\bibitem{ALIF1}
M.~Alif and J.-P. Gossez, \emph{On the {F}u\v c\'\i k spectrum with indefinite
  weights}, Differential Integral Equations \textbf{14} (2001), no.~12,
  1511--1530. \MR{1859919 (2002f:34024)}

\bibitem{ANA}
A.~Anane, O.~Chakrone, and M.~Moussa, \emph{Spectrum of one dimensional
  {$p$}-{L}aplacian operator with indefinite weight}, Electron. J. Qual. Theory
  Differ. Equ. (2002), No. 17, 11. \MR{1942086 (2004b:34044)}

\bibitem{AR1}
M.~Arias and J.~Campos, \emph{Fu\v cik spectrum of a singular
  {S}turm-{L}iouville problem}, Nonlinear Anal. \textbf{27} (1996), no.~6,
  679--697. \MR{1399068 (97g:34029)}

\bibitem{ARI}
M.~Arias, J.~Campos, M.~Cuesta, and J.-P. Gossez, \emph{Asymmetric elliptic
  problems with indefinite weights}, Ann. Inst. H. Poincar\'e Anal. Non
  Lin\'eaire \textbf{19} (2002), no.~5, 581--616. \MR{1922470 (2003g:35169)}

\bibitem{AR3}
\bysame, \emph{An asymmetric {N}eumann problem with weights}, Ann. Inst. H.
  Poincar\'e Anal. Non Lin\'eaire \textbf{25} (2008), no.~2, 267--280.
  \MR{2396522 (2009b:35137)}

\bibitem{Con}
L.~Baffico, C.~Conca, and M.~Rajesh, \emph{Homogenization of a class of
  nonlinear eigenvalue problems}, Proc. Roy. Soc. Edinburgh Sect. A
  \textbf{136} (2006), no.~1, 7--22. \MR{2217505 (2007b:35021)}

\bibitem{DF1992}
Andrea Braides, Valeria Chiad{\`o}~Piat, and Anneliese Defranceschi,
  \emph{Homogenization of almost periodic monotone operators}, Ann. Inst. H.
  Poincar\'e Anal. Non Lin\'eaire \textbf{9} (1992), no.~4, 399--432.
  \MR{1186684 (94d:35014)}

\bibitem{Zua2}
C.~Castro and E.~Zuazua, \emph{High frequency asymptotic analysis of a string
  with rapidly oscillating density}, European J. Appl. Math. \textbf{11}
  (2000), no.~6, 595--622. \MR{1811309 (2001k:34093)}

\bibitem{Zua1}
Carlos Castro and Enrique Zuazua, \emph{Low frequency asymptotic analysis of a
  string with rapidly oscillating density}, SIAM J. Appl. Math. \textbf{60}
  (2000), no.~4, 1205--1233 (electronic). \MR{1760033 (2001h:34117)}

\bibitem{CUE}
M.~Cuesta, D.~de~Figueiredo, and J.-P. Gossez, \emph{The beginning of the
  {F}u\v cik spectrum for the {$p$}-{L}aplacian}, J. Differential Equations
  \textbf{159} (1999), no.~1, 212--238. \MR{1726923 (2001f:35308)}

\bibitem{DAN}
E.~N. Dancer, \emph{On the {D}irichlet problem for weakly non-linear elliptic
  partial differential equations}, Proc. Roy. Soc. Edinburgh Sect. A
  \textbf{76} (1976/77), no.~4, 283--300. \MR{0499709 (58 \#17506)}

\bibitem{DAN-PER}
Norman Dancer and Kanishka Perera, \emph{Some remarks on the {F}u\v c\'\i k
  spectrum of the {$p$}-{L}aplacian and critical groups}, J. Math. Anal. Appl.
  \textbf{254} (2001), no.~1, 164--177. \MR{1807894 (2001k:35238)}

\bibitem{DGT}
A.~Derlet, J.-P. Gossez, and P.~Tak{\'a}{\v{c}}, \emph{Minimization of
  eigenvalues for a quasilinear elliptic {N}eumann problem with indefinite
  weight}, J. Math. Anal. Appl. \textbf{371} (2010), no.~1, 69--79. \MR{2660987
  (2011h:49069)}

\bibitem{DRA}
P.~Dr{\'a}bek, \emph{Solvability and bifurcations of nonlinear equations},
  Pitman Research Notes in Mathematics Series, vol. 264, Longman Scientific \&
  Technical, Harlow, 1992. \MR{1175397 (94e:47084)}

\bibitem{DRA-ROB}
Pavel Dr{\'a}bek and Stephen~B. Robinson, \emph{On the generalization of the
  {C}ourant nodal domain theorem}, J. Differential Equations \textbf{181}
  (2002), no.~1, 58--71. \MR{1900460 (2003g:35170)}

\bibitem{HT}
Siham El~Habib and Najib Tsouli, \emph{On the spectrum of the {$p$}-{L}aplacian
  operator for {N}eumann eigenvalue problems with weights}, Proceedings of the
  2005 {O}ujda {I}nternational {C}onference on {N}onlinear {A}nalysis (San
  Marcos, TX), Electron. J. Differ. Equ. Conf., vol.~14, Southwest Texas State
  Univ., 2006, pp.~181--190 (electronic). \MR{2316116 (2008b:35079)}

\bibitem{FUC}
Svatopluk Fu{\v{c}}{\'{\i}}k, \emph{Boundary value problems with jumping
  nonlinearities}, \v Casopis P\v est. Mat. \textbf{101} (1976), no.~1, 69--87.
  \MR{0447688 (56 \#5998)}

\bibitem{FUCIKLIBRO}
Svatopluk Fu{\v{c}}{\'{\i}}k and Alois Kufner, \emph{Nonlinear differential
  equations}, Studies in Applied Mechanics, vol.~2, Elsevier Scientific
  Publishing Co., Amsterdam, 1980. \MR{558764 (81e:35001)}

\bibitem{GGP}
T.~Godoy, J.-P. Gossez, and S.~Paczka, \emph{On the antimaximum principle for
  the {$p$}-{L}aplacian with indefinite weight}, Nonlinear Anal. \textbf{51}
  (2002), no.~3, 449--467. \MR{1942756 (2003i:35221)}

\bibitem{1DIM}
Juli\'an Fern\'andez~Bonder, Juan P.~Pinasco and Ariel~M. Salort,
  \emph{Eigenvalue homogenization for quasilinear elliptic equations with different boundary conditions, preprint},  arXiv:1208.5744v1
  (2012).

\bibitem{NDIM}
\bysame, \emph{Convergence rate for quasilinear eigenvalue homogenization, preprint}, arXiv:1211.0182v1 (2012).

\bibitem{Kenig}
Carlos~E. Kenig, Fanghua Lin, and Zhongwei Shen, \emph{Convergence rates in
  $l^2$ for elliptic homogenization problems, to appear in archive for rational
  mechanics and analysis}, arXiv:1103.0023v1 (2011).

\bibitem{LIAM}
Liamidi Leadi and Aboubacar Marcos, \emph{On the first curve in the {F}u\v cik
  spectrum with weights for a mixed {$p$}-{L}aplacian}, Int. J. Math. Math.
  Sci. (2007), Art. ID 57607, 13. \MR{2365740 (2008h:35094)}

\bibitem{RE1}
R.~Lemmert and W.~Walter, \emph{Singular nonlinear boundary value problems},
  Appl. Anal. \textbf{72} (1999), no.~1-2, 191--203. \MR{1775441 (2001d:34040)}

\bibitem{Mazya}
Vladimir~G. Maz'ja, \emph{Sobolev spaces}, Springer Series in Soviet
  Mathematics, Springer-Verlag, Berlin, 1985, Translated from the Russian by T.
  O. Shaposhnikova. \MR{817985 (87g:46056)}

\bibitem{Ol}
O.~A. Ole{\u\i}nik, A.~S. Shamaev, and G.~A. Yosifian, \emph{Mathematical
  problems in elasticity and homogenization}, Studies in Mathematics and its
  Applications, vol.~26, North-Holland Publishing Co., Amsterdam, 1992.
  \MR{1195131 (93k:35025)}

\bibitem{PER}
Kanishka Perera, \emph{On the {F}u\v c\'\i k spectrum of the
  {$p$}-{L}aplacian}, NoDEA Nonlinear Differential Equations Appl. \textbf{11}
  (2004), no.~2, 259--270. \MR{2210289 (2006m:35270)}

\bibitem{RYN}
Bryan~P. Rynne, \emph{The {F}u\v c\'\i k spectrum of general
  {S}turm-{L}iouville problems}, J. Differential Equations \textbf{161} (2000),
  no.~1, 87--109. \MR{1740358 (2000j:34031)}

\end{thebibliography}

\end{document}